\documentclass[a4paper,12pt]{amsart}
\usepackage[utf8]{inputenc}
\usepackage[english]{babel}

\renewcommand{\geq}{\geqslant}
\renewcommand{\leq}{\leqslant}

\usepackage{amsmath,amssymb,amsthm}  
\usepackage{amsfonts}  
\usepackage{bbm}    

\usepackage{bm}
\usepackage{hyperref}
\usepackage{accents}
\usepackage{mathrsfs}
\usepackage{mathtools}
\usepackage{fixmath}
\usepackage{fullpage}
\usepackage{paralist}
\usepackage{stmaryrd}

\usepackage{booktabs}
\usepackage{array}

\usepackage{pgf,tikz}
\usetikzlibrary{arrows}
\usetikzlibrary{patterns}
\usetikzlibrary[positioning]
\usetikzlibrary[fit]
\usetikzlibrary{shapes.geometric}
\usetikzlibrary{calc}

\usepackage{overpic}

\RequirePackage{etex}
\RequirePackage{etoolbox}
\RequirePackage{thmtools}
\RequirePackage{environ}
\RequirePackage{amsthm}
\RequirePackage{amssymb}
\newtheorem{theorem}{Theorem}
\newtheorem{proposition}[theorem]{Proposition}
\newtheorem{corollary}[theorem]{Corollary}

\newtheorem{theoremdefinition}[theorem]{Theorem-Definition}
\newtheorem{definition}[theorem]{Definition}
\theoremstyle{remark}
\newtheorem{example}[theorem]{Example}
\newtheorem{remark}[theorem]{Remark}

\setdefaultitem{$\triangleright$}{}{}{}

\usepackage{ifthen}
\usepackage[colorinlistoftodos,bordercolor=orange,backgroundcolor=orange!20,linecolor=orange,textsize=scriptsize]{todonotes}

\usepackage{xspace}

\DeclareSymbolFont{stmry}{U}{stmry}{m}{n}
\DeclareMathSymbol\mapsfromchar\mathrel{stmry}{"5B}

\DeclareMathOperator*{\tdet}{tdet}
\DeclareMathOperator*{\tper}{tdet}

\DeclareMathOperator*{\tcone}{tpos}

\DeclareMathOperator*{\conv}{conv}

\DeclareMathOperator*{\Sym}{Sym}

\DeclareMathAlphabet\mathbfcal{OMS}{cmsy}{b}{n}

\DeclareMathAlphabet{\mathbbold}{U}{bbold}{m}{n}

\newcommand{\R}[0]{\mathbb{R}}       %real    
\newcommand{\K}[0]{\mathbb {K}}                % Hardy field

\newcommand{\trop}[1][]{
\ifthenelse{\equal{#1}{}}{ \mathbb{T} }{ \mathbb{T}(#1) }%
}    %tropical

\newcommand{\transpose}[1]{#1^{\top}}

\DeclareMathOperator*{\val}{val}

\newcommand{\puiseuxP}[0]{\bm{\mathcal{P}}}   %Puiseux polyhedron
\newcommand{\p}[0]{\bm{p}}

\renewcommand{\a}[0]{\bm{a} }

\newcommand{\perturb}[2][]{
\ifthenelse{\equal{#1}{}}{ \tilde{\Ipert{#2}} }{ #2[#1] }%
}
\newcommand{\pert}[2][]{
\ifthenelse{\equal{#1}{}}{ \widetilde{\Ipert{#2}} }{ \Ipert{#2}[#1] }%
}

\newcommand{\Ipert}[1]{
\mathsf{#1}}

\newcommand{\puiseuxA}[0]{\bm{\mathcal{A}}} 
\newcommand{\f}[0]{\bm{f}}
\newcommand{\g}[0]{\bm{g}}
\newcommand{\vol}{\operatorname{vol}}

\newcommand{\Kplus}{\K_{\geq 0}}

\DeclareMathOperator*{\atconv}{tconv}

\DeclareMathOperator*{\Log}{Log}

\newlength{\mytemplen}

\title{A tropical isoperimetric inequality}

\author{Jules Depersin
\and St\'ephane Gaubert
\and Michael Joswig} 

\address[Jules Depersin]{\'Ecole Polytechnique, 91128 Palaiseau Cedex France}
\email{jules.depersin@polytechnique.edu}
\address[St{\'e}phane Gaubert]{INRIA and CMAP, \'Ecole Polytechnique, CNRS, Universit\'e Paris Saclay, 91128 Palaiseau Cedex France}
\email{stephane.gaubert@inria.fr}
\address[Michael Joswig]{%
 Institut f{\"u}r Mathematik,
 TU Berlin, 
 Straße des 17. Juni 136, 10623 Berlin, Germany}
\email{joswig@math.tu-berlin.de}
\thanks{S.~Gaubert is partially supported by the PGMO program of Fondation Math\'ematique Jacques Hadamard and EDF}
\thanks{M.~Joswig is partially supported by Einstein Foundation Berlin and DFG (Priority Program 1489 and Collaborative Research Center TRR 109).  This work was started while he was a CNRS INSMI visiting professor at CMAP, \'Ecole Polytechnique, UMR 7641 and IMJ, Universit\'e Pierre et Marie Curie, UMR 7586.}

\DeclareMathOperator\tdist{tdist}
\DeclareMathOperator\tvol{tvol}
\DeclareMathOperator\qvol{qtvol}
\DeclareMathOperator\tdiam{tdiam}
\newcommand\id{{\operatorname{id}}}
\newcommand\1{{\mathbf 1}}

\begin{document}

\begin{abstract}
We introduce tropical analogues of the notion of volume of polytopes, leading to a tropical version of the (discrete) classical isoperimetric inequality.
  The planar case is elementary, but a higher-dimensional generalization leads to an interesting class of ordinary convex polytopes, characterizing the equality case in the isoperimetric inequality. 
  This study is motivated by open complexity questions concerning linear optimization and its tropical analogs. 
\end{abstract}

\maketitle

\section{Introduction}
\noindent
The classical isoperimetric inequality states that the bounded planar region with given perimeter which maximizes the area is the circular disk.
Its discrete version, from which a proof of the smooth result can be derived, says that the triangle with fixed perimeter which maximizes the area is equilateral; see \cite{Blasjo:2005} for a nice survey.
A minor variation of the same problem asks to maximize the area for fixed diameter (instead of fixed perimeter), and it has the same result.
The tropical analog is a statement about the tropical convex hull of three points in the plane.
Our main contribution is a generalization of that result to arbitrary dimensions.

This work is motivated by research on delicate complexity issues related with classical linear programming.
In \cite{ABGJ:1405.4161} a family of linear programs was constructed which exhibits central paths with unusually large total curvature.
These linear programs provide counter-examples to a ``continuous analog of the Hirsch conjecture'' by Deza, Terlaky and Zinchenko~\cite{DTZ08}.
The key idea in \cite{ABGJ:1405.4161} was to obtain a lower bound for the total curvature of the central path by means of a piecewise-linear curve which can be associated with the tropicalization of  linear program.
In this way discrete notions of curvature, which make sense from a tropical geometry perspective, give rise to non-trivial metric estimates for classical curves.
This lead us to further investigate aspects of tropical geometry in the spirit of discrete differential geometry \cite{DDG}.
In particular, we are interested in tropical versions of the isoperimetric inequalities.

Tropical linear algebra is concerned with $(\max,+)$- or $(\min,+)$-analogs of classical linear algebra.
A \emph{tropical polyhedral cone} is a set of points in $\R^d$ which are tropical linear combinations of finitely many generators.
Its image under the projection modulo the all-ones vector is a \emph{tropical polytope}; see \cite{cgq02}, \cite{Tropical+Book} and the references there for an overview of the theory.
The search for a tropical analogue of volume lead us to propose a new notion which captures the metric intuition of tropical polytopes well enough.
Our main result shows that the tropical simplices which maximize this tropical volume for fixed tropical diameter are convex in the ordinary sense, i.e., these are the \emph{polytropes} studied in \cite{JoswigKulas:2010}.
The polytropes form the combinatorial building blocks of tropical convexity \cite[\S5.2]{Tropical+Book}.
In combinatorial optimization they arise naturally in the study of shortest path algorithms \cite[\S8.3]{Schrijver03:CO_A}, \cite{Sergeev:2007}, \cite{JoswigLoho:2016}.
Furthermore, they are isomorphic to tropical eigenspaces (see e.g.~\cite[Chapter~4]{Butkovic:2010}), play a role in the theory of semigroups \cite{JohnsonKambites:2015} and occur in statistical ranking \cite{SturmfelsTran:2013}.
It is known that, up to symmetry, there is precisely one generic combinatorial type of polytropes in $\R^3/\R\1$~\cite{JoswigKulas:2010}.
The generic polytropes in $\R^4/\R\1$ were classified in~\cite{JiminezDelaPuente:1205.4162} (see also~\cite{Tran:1310.2012}); there are precisely six types.
However, it turns out that, at least in these dimensions, only one generic type maximizes the tropical volume for fixed tropical diameter.

A second approach to obtain a tropical analogue of volume is to employ the ``dequantization'' method 
{\cite{kolomaslov}}, thinking of a tropical polytope as a log-limit of a family of
classical polytopes, and defining the limit of the normalized volumes of these polytopes.  We shall see that the volume
obtained in this way also has several good properties.  For instance, it turns out to be an idempotent measure on the
space of generic tropical polytopes and easy to compute. However, this yields a degenerate isoperimetric inequality.

\section{Tropical distance and volume}\label{sec-tropd}
\noindent
For two points $v,w\in\R^d$ the \emph{tropical distance}
is the number
\[
\begin{split}
  \tdist(v,w) \ :=& \ \max\bigl\{ (v_i-w_i) \mid i\in[d] \bigr\} - \min\bigl\{ (v_i-w_i) \mid i\in[d] \bigr\}\\
  =& \ \max_{i,j\in[d]} \bigl| v_i - w_i + w_j - v_j\bigr| \enspace .
\end{split}
\]
This number  was shown in~\cite{cgq02} to play the role of the Euclidean
distance in the tropical setting. It is a special instance of Hilbert's projective metric.
We have
\begin{equation}\label{eq:tdist:add}
  \begin{aligned}
    \tdist(\1+v,w) \  &= \ \tdist(v,w) \quad \text{and}\\
    \tdist(u+v,u+w) \ &= \ \tdist(v,w) \quad \text{for all } u\in \R^d \enspace.
  \end{aligned}
\end{equation}
In particular, this entails that $\tdist$ induces a metric on the \emph{tropical projective torus} $\R^d/\R\1$.
Moreover, we have
\begin{equation}\label{eq:tdist:scale}
  \tdist(\lambda\cdot v,\lambda\cdot w) \ = \ |\lambda|\cdot\tdist(v,w) \quad \text{for all } \lambda \in \R \enspace .
\end{equation}
This distance function is valid for both, $\min$ and $\max$, as the tropical addition, denoted by $\oplus$ in the sequel.

Now let $A=(a_{ij})\in\R^{d\times d}$ be a square matrix.
We write $a_{i\cdot}$ for the $i$th row and $a_{\cdot k}$ for the $k$th column of $A$.
The \emph{tropical diameter} of $A$ is the maximum
\[
\tdiam A \ := \ \max_{i,j\in[d]} \tdist(a_{i\cdot},a_{j\cdot}) \ = \ \max_{i,j,k,\ell}  \bigl| a_{ik} - a_{i\ell} + a_{j\ell} - a_{jk}\bigr| \enspace .
\]
Notice that the tropical diameter is the same as the diameter of a complete metric graph on $d$ nodes with non-negative edge lengths.
Further, the tropical diameters of a square matrix and its transpose agree.
\begin{example}\label{exmp:unit_matrix}
  The tropical diameter of the ordinary $d{\times}d$-unit matrix, with ones on the diagonal and zeros elsewhere, equals two.
\end{example}
Observe that the tropical diameter does not change if the rows or columns of $A$ are permuted.
\begin{definition}
We now define the \emph{tropical volume} of $A$ as the expression
\[
\tvol A \ := \ \left|\bigoplus_{\sigma\in\Sym(d)}\sum a_{i,\sigma(i)} - \bigoplus_{\tau\in(\Sym(d)-\sigma_{\text{opt}})}\sum a_{i,\tau(i)}\right| \enspace ,
\]
where $\sigma_{\text{opt}}$ is an optimal solution of the first optimization in the above.
\end{definition}
In other words, $\sigma_{\text{opt}}$ is a permutation for which $\sum a_{i,\sigma_{\text{opt}}(i)}$ coincides with the tropical determinant of $A$.
Like the tropical diameter also the tropical volume is insensitive to transposing the matrix $A$ or to any reordering of its rows or columns.
The tropical volume can be computed in $O(d^3)$ time \cite[\S5.4.1]{Assignment}.

Unlike the tropical diameter, which is an established notion, our definition of the tropical volume is new, at least under this name.
Since our results below rely on this notion in a crucial way, a few words are in order.
The classical volume has its foundation in measure theory, and the classical determinant yields the (normalized) volume of a simplex.
Tropical polytopes arise by ``dequantization'' of classical polytopes, or,
if one prefers, as images of ordinary convex polytopes over real Puiseux series under the valuation map; this was first observed by Develin and Yu \cite{DevelinYu:2007},
and this is the point of departure of \cite{ABGJ:1405.4161}.
This leads to a notion of dequantized tropical volume with several good properties, but we defer the discussion to Section~\ref{sec-dequantize}.

We prefer our definition of the tropical volume, $\tvol$, since it leads to more interesting isoperimetric problems.
The fact that it captures an essential metric property of tropical polytopes can be seen from the following observations.
First, the tropical volume is a higher-dimensional generalization of the tropical distance function: indeed, in the linear case $d=2$ the tropical diameter and the tropical volume agree, i.e., $\tdiam A = \tvol A$ if $A$ is a $2{\times}2$-matrix.
More importantly, the tropical volume provides a measure of non-singularity:
it vanishes if and only if the rows (or the columns) of $A$ are contained in a tropical hyperplane \cite[Lemma~5.1]{RichtergebertSturmfelsTheobald:2005}. 
In terms of statistical physics, the tropical volume is an energy gap, which appeared in the analysis of a non-standard optimal assignment algorithm
by Kosowsky and Yuille. Their key result~\cite[Theorem~9]{KosowskyYuille:1994}
estimates the speed of convergence by an increasing function of the energy
gap. Characterizing matrices with a maximal energy gap, knowing
bounds on their entries, is precisely a tropical isoperimetric
problem.

We call two square matrices \emph{equivalent} if they can be transformed into one other by row and column permutations or by operations as in (\ref{eq:tdist:add}).
Up to reordering the rows and columns we may assume that the identity permutation attains the tropical determinant.
Since neither the tropical diameter nor the tropical volume changes if we translate each column by the same vector, we may assume that the first column is the vector $(1,0,0,\dots,0)$.
Further, we can (ordinarily) add any multiple of $\1$ to each column without changing the tropical diameter, the tropical determinant or the tropical volume.
Thus we may assume that each column, except for the first, begins with a zero.
We call a matrix \emph{$\max$-standard} if the identity is an optimal permutation and if the first row and column read $(1,0,0,\dots,0)$.
Each square matrix is equivalent to a $\max$-standard matrix.
In view of the Example~\ref{exmp:unit_matrix} we will subsequently normalize the tropical diameter to two.
\begin{theorem}[Tropical isodiametric inequality]\label{thm:iso}
  Assume that $\oplus=\max$ is the tropical addition.
  Let $A\in\R^{d\times d}$ be a matrix with tropical diameter two.
  Then the tropical volume does not exceed two.
  Moreover, if $\tvol A=2$ then $A$ is equivalent to a $\max$-standard matrix whose coefficients $a_{ij}$ satisfy the following conditions:
  \begin{compactenum}
  \item[(i)]  $-1 \leq a_{ij}\leq 1$,
  \item[(ii)] $a_{ii} = 1$,
  \item[(iii)] $a_{ji} = - a_{ij}$ for $i\neq j$, and
  \item[(iv)] $-1 \leq a_{ij}+a_{jk}+a_{ki} \leq 1$ for $i,j,k$ distinct.
  \end{compactenum}
  Conversely, the tropical diameter and the tropical volume of each standard matrix satisfying these four conditions both equal two.
\end{theorem}
\begin{proof}
  Let $A$ be any square matrix with tropical diameter two.
  
  We need to examine the coefficients outside the first row and the first column.
  Since $A$ has tropical diameter two, none of these coefficients can exceed one.

  The Example~\ref{exmp:unit_matrix} shows whose tropical diameter and volume both yield two.
  Hence we may assume that $\tvol A\geq 2$.
  For any permutation $\sigma\in\Sym(d)$ we abbreviate $A|_\sigma := \sum a_{i,\sigma(i)}$.
  With this notation we have
  \begin{equation}\label{eq:iso:tvol}
    \tvol A \ = \ A|_\id - \max_{\tau\neq\id} A|_\tau \ \geq 2 \enspace .
  \end{equation}
  Now consider the transposition $\rho=(1\ i)$ for any $i\geq 2$.
  Then (\ref{eq:iso:tvol}) forces that
  \[
  A|_\id - A|_\rho \ = \ a_{11}+a_{ii}-a_{1i}-a_{i1} \ = \ 1 + a_{ii} - 0 - 0 \ \geq \ 2 \enspace ,
  \]
  which gives $a_{ii}=1$ since no coefficient of $A$ is larger than one.
  Summing up the discussion so far, our matrix $A$ has the following shape:
  \[
  \begin{pmatrix}
    1      & 0 & 0    & \cdots & 0 \\
    0      & 1 & a_{23}& \cdots & a_{2d} \\
    0      & a_{32}& 1 & \ddots  & \vdots \\
    \vdots & \vdots & \ddots & \ddots & a_{d-1,d} \\
    0      & a_{d2} & \cdots & a_{d,d-1} & 1
  \end{pmatrix} \enspace .
  \]

  Next we consider the transposition $\tau=(i\ j)$ for distinct $i,j\neq 1$.
  Again from (\ref{eq:iso:tvol}) we get
  \[
  A|_\id - A|_\tau \ = \ a_{ii} + a_{jj} - a_{ij} - a_{ji} \ = \ 2 - a_{ij} - a_{ji} \geq \ 2
  \]
  and thus $a_{ij}+a_{ji} \leq 0$.
  However, we have
  \[
  |2 - (a_{ij} + a_{ji})| \ \leq \ \tdist(a_{\cdot i},a_{\cdot j}) \ \leq \ 2 \enspace ,
  \]
  which implies $a_{ij}=-a_{ji}$.
  
  Finally, we consider the 3-cycle $\sigma=(i\ j\ k)$ for pairwise distinct $i,j,k$.
  Applying (\ref{eq:iso:tvol}) for a third time gives
  \[
  A|_\id - A|_\sigma \ = \ 3 - a_{ij} - a_{jk} - a_{ki} \ \geq 2 \enspace .
  \]
  We arrive at $a_{ij} + a_{jk} + a_{ki} \leq 1$.
  Since this argument holds for $i,j,k$ arbitrary with (ii) we also have
  \[
  a_{ij} + a_{jk} + a_{ki} \ = \ - a_{ji} - a_{kj} - a_{ik} \ = \ -(a_{ik} + a_{kj} + a_{ji}) \ \geq \ -1 \enspace .
  \]

  So far we have shown that any matrix of tropical diameter two and tropical volume greater than or equal to two satisfies the four conditions claimed.
  To prove all remaining claims we need to show that each matrix which satisfies the conditions (i) through (iv) has tropical diameter and tropical volume two.
  It is easy to see that for a standard matrix, in fact, the first three properties force that the tropical diameter attains the desired value.

  We need to prove that the tropical volume of a matrix which satisfies all four conditions is exactly two.
  Let $\sigma=(\sigma_1 \ \sigma_2 \ \ldots \ \sigma_\ell)\in\Sym(d)$ be an arbitrary $\ell$-cycle for $\ell\geq 3$ with pairwise distinct indices $\sigma_i$.
  Then from (iv) and (iii) we get
  \[
  a_{\sigma_1,\sigma_2} + a_{\sigma_2,\sigma_3} \ \leq \ 1 - a_{\sigma_3,\sigma_1} \ = \ 1 + a_{\sigma_1,\sigma_3}
  \]
  and thus
  \[
  a_{\sigma_1,\sigma_2} + a_{\sigma_2,\sigma_3} + \dots + a_{\sigma_{\ell-1},\sigma_\ell} + a_{\sigma_\ell,\sigma_1} \ \leq \ (\ell-3)+1 \ = \ \ell-2
  \]
  by induction.
  It follows that $A|_\sigma \leq (d-\ell)+(\ell-2) = d-2$.
  If now $\sigma'$ is an arbitrary non-identity permutation it decomposes into $k\geq 1$ disjoint cycles, and it follows that
  \[
  A|_{\sigma'} \ \leq \ d-2k \quad \text{and thus} \quad A|_\id - A|_{\sigma'} \ \geq \ 2k \enspace .
  \]
  For the transposition $(1\ 2)$ we obtain $A|_{(1\ 2)}=d-2$ by a direct computation, and this finally yields $\tvol A = A_\id - A_{(1\ 2)} = 2$.
\end{proof}

We call a matrix \emph{$\min$-standard} if the identity is an optimal permutation and if the first row and column read $(0,1,1,\dots,1)$.
Each square matrix is equivalent to a $\min$-standard matrix.

\begin{corollary}[Tropical isodiametric inequality]\label{cor:iso}
  Assume that $\oplus=\min$ is the tropical addition.
  Let $B\in\R^{d\times d}$ be a matrix with tropical diameter and tropical volume two.
  Then $B$ is equivalent to a $\min$-standard matrix whose coefficients $b_{ij}$ satisfy the following conditions:
  \begin{compactenum}
  \item[(i)]  $0 \leq b_{ij}\leq 2$,
  \item[(ii)] $b_{ii} = 0$,
  \item[(iii)] $b_{ij} + b_{ji} = 2$ for $i\neq j$, and
  \item[(iv)] $2 \leq b_{ij}+b_{jk}+b_{ki} \leq 4$ for $i,j,k$ distinct.
  \end{compactenum}
  Conversely, the tropical diameter and the tropical volume of each standard matrix satisfying these four conditions both equal two.
\end{corollary}
\begin{proof}
  Let $A$ be a $\max$-standard matrix as in Theorem~\ref{thm:iso}.
  Then $B:=\1-A$, where $\1$ is the all-ones matrix, is $\min$-standard, and it satisfies the claim.
  Going to the negative exchanges $\max$ and $\min$.
\end{proof}
In the sequel we will be concerned with non-negative $d{\times}d$-matrices $B$ which satisfy the conditions (ii), (iii) and (iv) in Corollary~\ref{cor:iso}.
We call any matrix which is equivalent to such a matrix \emph{tropically near-isodiametric} (with respect to $\min$).
The matrix is \emph{tropically isodiametric} if additionally the tropical diameter and the tropical volume are equal to two.
\begin{proposition}\label{prop:B2}
  Let $B\in\R^{d\times d}$ be a tropically near-isodiametric matrix with respect to $\oplus=\min$ as the tropical addition.
  Then the matrix equation $B\odot_{\min}B=B$ holds.
\end{proposition}
Notice that here we do \emph{not} assume $B$ to be standard.
That is, we do not specify the first row and column.
Also we do not require the upper bound in property (i), such that and the coefficients may be larger than two.
\begin{proof}
  Let $c_{ij}$ be the coefficient of the matrix $C:=B\odot_{\min}B$ in the $i$th row and the $j$th column.
  The value
  \begin{equation}\label{eq:B2:min}
    c_{ij} \ = \ \min(b_{i1}+b_{1j},b_{i2}+b_{2j},\dots,b_{id}+b_{dj})
  \end{equation}
  is the $\min$-tropical scalar product of the $i$th row $b^{i\cdot}$ with the $j$th column $b^{\cdot j}$ of $B$.
  We will show that $C=B$.

  First note that $c_{ij}\geq 0$ for all $i$ and $j$ since each coefficient of $B$ are assumed to be non-negative.
  Moreover, from (ii) we get
  \begin{equation}\label{eq:B2:ineq}
    c_{ij} \ \leq \ b_{ii}+b_{ij} \ = \ b_{ij}+b_{jj} \ = \ b_{ij} \enspace .
  \end{equation}
  Specializing $i=j$ in (\ref{eq:B2:ineq}) forces $c_{ii}=0$, and this means that the diagonal entries of $C$ and $B$ agree.
  Now let $i,j,k$ be pairwise distinct.
  In this case the inequality (iv) gives
  \[
  b_{ij} \ = \ 2-b_{ji} \ \leq \ b_{ik}+b_{kj} \ \leq \ 4-b_{ji} \ = \ 2-b_{ij}
  \]
  if combined with (iii).
  We conclude that for $i\neq j$ the minimum in (\ref{eq:B2:min}) as attained (twice) at $b_{ii}+b_{ij} = b_{ij}+b_{jj} = b_{ij}$.
  This completes the proof of our claim $C=B$.
\end{proof}

The \emph{conical tropical convex hull} of a $d{\times}m$-matrix $M$,
denoted as $\tcone M$, is the set $\{M\odot x\mid x\in\R^m\}$.
Since this is a homogeneous notion we usually consider $\tcone M$ as a subset of the tropical projective torus $\R^d/\R\1$.
The following statement is phrased without an explicit reference to a tropical addition.
It works in both cases.
\begin{corollary}\label{cor:polytrope}
  Let $M\in\R^{d\times d}$ be near-isodiametric.
  Then the conical tropical convex hull of the columns (or the rows) of $M$ is convex in the ordinary sense, i.e., it is a polytrope.
\end{corollary}
\begin{proof}
  Let us take $\min$ as the tropical addition here.
  We may assume that $M=B$ as in Corollary~\ref{cor:iso}.
  Then Proposition~\ref{prop:B2} says that $B$ agrees with its \emph{Kleene star}
  \[
  B^* \ := \ I \oplus B \oplus B^{\odot 2} \oplus B^{\odot 3} \odot \cdots \enspace ,
  \]
  where $I$ is the $\min$-tropical identity matrix, which has zero coefficients on the diagonal and $\infty$ otherwise.
  It is known that in this case $\tcone(B)=\tcone(B^*)$ is a weighted digraph polyhedron and thus a polytrope.
  For a proof see, e.g., \cite[Theorem~8.3.11]{Butkovic:2010} or \cite[Theorem~2.1]{delaPuente:2013}.
\end{proof}
Note that, even for general $B$, the Kleene star $B^*$ is the shortest path matrix for the digraph on $d$ nodes whose weights are given by the coefficients of $B$; see also \cite[\S3.4]{JoswigLoho:2016}.
A polytrope is \emph{isodiametric} if it arises from a tropically isodiametric matrix via Corollary~\ref{cor:polytrope}.
\begin{example}\label{exmp:isoplanar}
  For $d=3$ any isodiametric $\min$-standard matrix looks like
  \[
  B(\lambda) \ = \
  \begin{pmatrix}
    0 & 1 & 1 \\
    1 & 0 & \lambda \\
    1 & 2-\lambda & 0
  \end{pmatrix} \enspace ,
  \]
  where $0\leq\lambda\leq 2$.
  The planar polytropes which arise as the $\min$-tropical convex hulls of the columns of the matrices $B(\lambda)$ are shown in Figure~\ref{fig:isoplanar} for various values of $\lambda$.
  The red points mark the non-redundant generators, i.e., the columns, while the white points are the pseudo-vertices, generically.
  Going from $\min$ to $\max$ means to interchange the roles of the red and the white points.
  For the non-generic cases $\lambda=0$ and $\lambda=2$ the non-redundant generators for $\min$ and $\max$ agree.
\end{example}

\newcommand\isoplanar[1]{%
\begin{tikzpicture}[x  = {(1cm,0cm)},
                    y  = {(0cm,1cm)},
                    z  = {(0cm,0cm)},
                    scale = 0.9,
                    color = {lightgray}]

  \tikzstyle{pseudostyle} = [fill=white, draw=black, thin]
  \tikzstyle{vertexstyle} = [fill=red, draw=black, thin]
  \tikzstyle{linestyle} = [draw=black, thick]
  \tikzstyle{facestyle} = [fill=blue!60!green!40]

  \coordinate (v0) at ($ (-1,-1) + (0,2-#1) $);
  \coordinate (v1) at ($ (1,-1) + (0,2-#1) $);
  \coordinate (v2) at ($ (1,1) - (2-#1,0) $);
  \coordinate (v3) at (1,1);
  \coordinate (v4) at ($ (1,-1) - (2-#1,0) $);
  \coordinate (v5) at (-1,-1);

  \filldraw[facestyle,linestyle] (v0) \foreach \i in {2,3,1,4,5}{ -- (v\i) } -- cycle;

  \foreach \i in {1,2,5} { \filldraw[pseudostyle] (v\i) circle (1.5pt); }
  \foreach \i in {0,3,4} { \filldraw[vertexstyle] (v\i) circle (2pt); }
\end{tikzpicture}}

\begin{figure}[hbt]\centering
  \setlength{\tabcolsep}{18pt}
  \begin{tabular}{ccccc}
    \isoplanar{0} & \isoplanar{1/2} & \isoplanar{1} & \isoplanar{3/2} & \isoplanar{2} \\
    $\lambda=0$ & $\lambda=\frac{1}{2}$ & $\lambda=1$ & $\lambda=\frac{3}{2}$ & $\lambda=2$
  \end{tabular}
  \caption{Isodiametric polytropes in $\R^3/\R\1$}
  \label{fig:isoplanar}
\end{figure}

As shown in Example~\ref{exmp:isoplanar} the isodiametric polytropes (with fixed diameter) depend on one real parameter which is, moreover, bounded between zero and two.
In the general case we have $((d-1)^2-(d-1))/2=(d^2-3d)/2+1$ free parameters which are constrained by linear inequalities.
That is, the isodiametric polytropes in $\R^d/\R\1$ are parameterized by a convex polytope Iso$(d)$ of that dimension.
While this is naturally embedded in a real vector space of dimension $d^2$, we usually look at its faithful projection into the coordinate directions given by the coefficients $b_{ij}$ for $2\leq i<j \leq d-1$.
Notice that, up to this projection, Iso$(d)$ is contained in the dilate $2\cdot[0,1]^{d-1}$ of the unit cube by a factor of two.
The polytope Iso$(3)$ is the segment $[0,2]$.

\section{Combinatorics of near-isodiametric polytropes}
\noindent
Let $B=(b_{ij})$ be an $d{\times}d$-matrix which is tropically near-isodiametric with respect to $\min$.
We want to analyze the polytrope $P:=\tcone(B)$ seen as an ordinary convex polytope in $\R^{d-1}$.
That latter space is identified with $\R^d/\R\1$ via the map
\(
(x_1,x_2,\dots,x_d) \mapsto (x_2-x_1,\dots,x_d-x_1) 
\).
Our point of departure is the exterior description
\begin{equation}\label{eq:wdp}
  P(B) \ = \ \bigl\{ x\in\R^d \bigm| x_i-x_j \leq b_{ij} \text{ for } i\neq j \bigr\}
\end{equation}
as a \emph{weighted digraph polyhedron}; see \cite[\S5.2]{Tropical+Book} and \cite{JoswigLoho:2016}.
Since $B=B^*$ is a Kleene star all these inequalities are tight; this classical result follows, e.g., from \cite[2.3.3]{Gallai:1958}.
The one-dimensional lineality space of $P(B)$ is $\R\1$.
Let us set $P'(B):=P(B)/\R\1$.

\begin{proposition}\label{prop:facets}
  Assume that the inequalities (iv) in Corollary~\ref{cor:iso} are strict, i.e.,
  \[
    2<b_{ij}+b_{jk}+b_{ki}<4 \quad \text{for all } i,j,k \enspace .
  \]
  Then $P'(B)$ is an ordinary polytope of dimension $d-1$ with exactly $d(d-1)$ facets.
\end{proposition}
\begin{proof}
  Fix $i$ and $j$ distinct.
  We will show that the inequality $x_i-x_j \leq b_{ij}$ in (\ref{eq:wdp}) defines a facet of $P'(B)$ if and only if $2<b_{ij}+b_{jk}+b_{ki}<4$ for all $k$.

  Assume first that there is an index $k$ such that $b_{ij}+b_{jk}+b_{ki}=4$.
  Then we can add the inequalities $x_i-x_k \leq b_{ik}$ and $x_k-x_j \leq b_{kj}$ to obtain
  \[
    x_i - x_j \ \leq \ b_{ik}+b_{kj} \ \stackrel{\text{(iii)}}{=} \ 4-b_{ki}-b_{jk} \ = \ b_{ij} \enspace .
  \]
  That is, if we assume equality, then the inequality $x_i - x_j \leq b_{ij}$ is implied by other valid inequalities.
  Symmetrically, the equality $2=b_{ij}+b_{jk}+b_{ki}$ forces
  \[
    x_j - x_i \ \leq \ b_{jk}+b_{ki} \ = \ 2-b_{ij} \ \stackrel{\text{(iii)}}{=} \ b_{ji} \enspace .
  \]

  To prove the converse suppose first that $b_{ij}+b_{jk}+b_{ki}<4$ holds for all $k$.
  Then the number
  \[
    \epsilon \ := \ \min \biggl( 2, \ \min_k \bigl(4 - (b_{ij}+b_{jk}+b_{ki})\bigr) \biggr) \enspace ,
  \]
  is strictly positive.
  We will construct a point $x\in\R^d$ which satisfies all inequalities in (\ref{eq:wdp}) except for $x_i-x_j \leq b_{ij}$.
  Let us set
  \[
    x_i \ := \  b_{ij}+\epsilon \enspace, \quad x_j \ := \ 0 \enspace, \quad \text{and} \quad x_k \ := \ b_{kj} \ = \ 2-b_{jk}
  \]
  for all $k\neq i,j$.
  Then we find
  \[
    \begin{aligned}
      x_i - x_j \ &= \ b_{ij}+\epsilon \ > \ b_{ij} \enspace, \\
      x_j - x_i \ &= \ -b_{ij}-\epsilon \ \leq \ b_{ji}-2 < b_{ji} \enspace, \\
      x_i - x_k \ &= \ b_{ij}+\epsilon + b_{jk}-2 \ \leq \ 2 - b_{ki} \ = \ b_{ik} \enspace, \\
      x_k - x_i \ &= \ b_{kj} - b_{ij}-\epsilon \ \leq \ 2-b_{jk}-b_{ij} \ \leq \ b_{ki} \enspace, \\
      x_j - x_k \ &= \ b_{jk}-2 \ < \ b_{jk} \enspace, \\
      x_k - x_j \ &=  \ b_{kj} \enspace, \\
      x_k - x_\ell \ &= \ b_{kj} + b_{j\ell}-2 \ \leq \ 2 - b_{\ell k} \ = \ b_{k\ell} \enspace,
    \end{aligned}
  \]
  where $\ell\neq i,j,k$, and this shows that the inequality $x_i-x_j \leq b_{ij}$ is not redundant.
\end{proof}
Any real $d{\times}n$-matrix $M$ induces a height function on the vertices of the product of simplices $\Delta_{d-1}\times\Delta_{n-1}$.
The induced regular subdivision is dual to (the covector decomposition of) the conical tropical convex hull of the columns of $M$; see \cite[\S5.2]{Tropical+Book} and \cite{JoswigLoho:2016}.
In the \emph{generic} case that subdivision is a triangulation.
Our main result says that there are isodiametric matrices which are generic.
\begin{theorem}
  For each $n\geq3$ there exist isodiametric matrices $B$ such that the polytrope $P'(B)$ is a simple ordinary polytope with exactly $\tbinom{2d}{d}$ vertices.
  In this case, the regular subdivision of $\Delta_{d-1}\times\Delta_{d-1}$ induced by $B$ is a triangulation.
\end{theorem}
\begin{proof}
  Let $B$ be an isodiametric $d{\times}d$-matrix for which the inequalities (iv) in Corollary~\ref{cor:iso} are strict and whose off-diagonal coefficients are strictly positive.
  By \ref{prop:facets} the ordinary polytope $P'(B)$ has $d(d-1)$ facets.
  Then there exists an $\epsilon>0$ such that, for each skew-symmetric $d{\times}d$-matrix $E=-\transpose{E}$ whose coefficients lie between $\pm\epsilon$, the ordinary sum $B+E$ is isodiametric, too, and the ordinary polytopes $P'(B)$ and $P'(B+E)$ are normally equivalent.
  It follows that $P'(B)$ must be a simple polytope.

  Now let $E$ be an arbitrary matrix with sufficiently small coefficients, but which is not necessarily skew symmetric but which has a zero diagonal.
  In this case the matrix $B+E$ is still near-isodiametric, but not necessarily isodiametric.
  Yet, $P'(B)$ is still normally equivalent to $P'(B+E)$.
  In particular $P'(B)$ and $P'(B+E)$ are combinatorially isomorphic.
  As we have sufficiently many free parameters for the choice of $E$ it follows that the regular subdivision of $\Delta_{d-1}\times\Delta_{d-1}$ induced by $B$ is a triangulation.
  Each such triangulation has $\tbinom{2d}{d}$ maximal cells, and these are dual to the vertices of $P'(B)$.
\end{proof}

\begin{example}
  The matrix
  \[
    B \ = \ \begin{pmatrix}
      0 & 1 & 1 & 1 \\
      1 & 0 & 5/4 & 3/4 \\
      1 & 3/4 & 0 & 5/4 \\
      1 & 5/4 & 3/4 &0
    \end{pmatrix}
  \]
  is a tropically isodiametric matrix, which is standard with respect to $\min$, and which is generic.
  The resulting polytrope $P'(B)$, shown in Figure~\ref{fig:maxiso4}, is combinatorially equivalent to the second example in \cite[Figure~2]{JoswigKulas:2010}.
 It belongs to Class~1 in the classification \cite[\S3.9]{JiminezDelaPuente:1205.4162}:  among the 12 ordinary facets there are three quadrangles, six pentagons and three hexagons, and there is no pair of adjacent hexagons.
\end{example}

\begin{figure}[hbt]
  \centering
  \includegraphics[width=.4\textwidth]{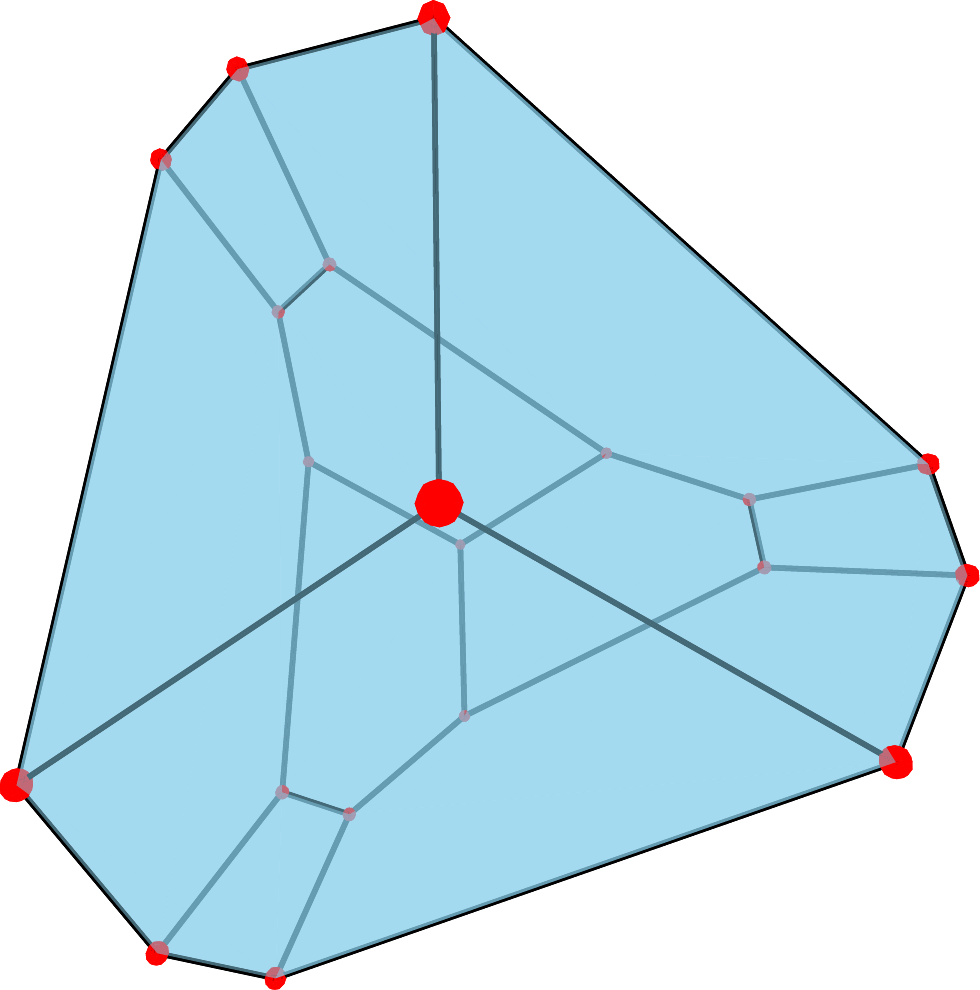}
  \caption{Isodiametric polytrope in $\R^4/\R1$ }
  \label{fig:maxiso4}
\end{figure}
\begin{remark}
  A computation with \texttt{polymake} \cite{DMV:polymake} reveals the following: 
  None of the other four combinatorial types of maximal polytropes for $n=4$ from \cite[Figure~2]{JoswigKulas:2010} admits a tropically isodiametric realization, nor does the sixth type, which was found by Jim\'{e}nez and de la Puente \cite[Example~21]{JiminezDelaPuente:1205.4162}.
\end{remark}

\section{Tropical volume arising from dequantization}\label{sec-dequantize}
We now investigate a different notion of ``volume'', which is also applicable to tropical geometry.
It arises from ``dequantization'', a term coined by Maslov for a procedure in which tropical objects are obtained as the log-limits of classical objects; see~\cite{kolomaslov}.
A related procedure is known as Viro's method~\cite{viro} in real algebraic geometry.

In this approach,
tropical numbers can be thought of as the images of the elements of a non-archimedean field under the valuation map.
{Following~\cite{alessandrini2013,ABGJ:1405.4161}, we will work with a field $\K$ of real valued functions in a real parameter $t$. We assume that $\K$ is a Hardy field whose elements are germs at infinity of real valued functions of  $t$ that are definable in a fixed polynomially bounded o-minimal structure. 
Then, the \emph{valuation} of a function $\f\in \K$ is defined by
\(
\val \f  \ := 
\lim_{t\to\infty}(\log t)^{-1}\log |\f(t)|
\).
The map $\val$ is a \emph{non-archimedean} valuation, meaning that
$\val 0=-\infty$,
\(
\val \f\g=\val\f+\val\g\),
and \(%
\val \f+\g \leq \max(\val \f, \val \g)\).
Moreover, the latter inequality becomes an equality if $\f,\g\in \Kplus$, the
subset of non-negative functions of $\K$. We will assume that every function $t^r$ with $r\in \R$ belongs to $\K$. Then, $\val$ yields a surjective morphism of semifields from $\Kplus$ to the tropical semifield $\trop$ (with ground set $\R\cup\{-\infty\}$ and maximum as addition). }

The notions of convex hull, polyhedra, etc., make sense over $\K$.
In particular, if $\puiseuxA=(\a_{ij})\in \K^{d\times m}$, we denote by  $\puiseuxP:=\conv\puiseuxA$ the polytope generated by the columns of $\puiseuxA${.}
By evaluating the matrix $\puiseuxA(t)=(\a_{ij}(t))$ at a real parameter $t$, we obtain a polytope $\puiseuxP_t:=\conv\puiseuxA(t)$, so $\puiseuxA$ encodes a parametric family of ordinary polytopes. 
We will denote by $\bar\puiseuxA$ the $(d+1)\times m$ matrix
obtained by adding an identically one row to the $d\times m$ matrix $\puiseuxA$, putting this new row at the top of the matrix. If $\puiseuxA$ has $m = d+1$ affinely independent columns, $\puiseuxP$ is a simplex with volume $(d!)^{-1}|\det \bar{\puiseuxA}|$. In general, the volume of $\puiseuxP$, denoted by $\vol \puiseuxP$ or $\vol \puiseuxA$, can be computed by triangulating the configuration of points given by columns of $\puiseuxA$ into simplices, just like over $\R$.

If $A=(a_{ij})\in \trop^{d\times m}$, we say that $\puiseuxA=(\a_{ij})\in \Kplus^{d\times n}$ is a \emph{lift} of $A$ if $\val \puiseuxA =A$. 
While we worked with tropical cones and projective coordinates
in Section~\ref{sec-tropd}, now it is more convenient to consider affine notions as follows.
We call \emph{affine tropical convex hull} of the columns of $A$, denoted as $\atconv A$, the set of vectors of the form $A\odot x$ where $x \in \trop^m$ is such
that $\max_i x_i=0$. We will now use the notation $\bar{A}$ to denote
the $(d+1)\times m$ matrix obtained by adding an identically zero row to $A$
(still on the top row).
Note also that if $A=\val \puiseuxA$,
we have $\val \bar \puiseuxA= \bar A$. 
In this way, when the entries of $A$ are finite, $\atconv A$ can be identified to the cross section
by the hyperplane $x_1=0$ of the conical tropical
convex hull, $\tcone \bar A$, defined in Section~\ref{sec-tropd}.

A result of Develin and Yu~\cite{DevelinYu:2007} implies that every affine tropical polytope $P$ is of the form $\val \puiseuxP$,
where $\puiseuxP= \conv\puiseuxA$ for some matrix $\puiseuxA$ with entries
in $\K$. 
This motivates the following definition of tropical
volumes of a polytope in terms of log-limits.

\begin{definition}
Let $A\in \trop^{d\times m}$.
The \emph{upper} and \emph{lower} \emph{dequantized tropical volumes} of $A$ are defined by
\begin{align*}
\qvol^+ A:= \sup\{\val \vol \puiseuxA\mid \val\puiseuxA = A\}\quad\text{and}\quad
\qvol^- A:= \inf\{\val \vol \puiseuxA\mid \val\puiseuxA = A\}\enspace,
\end{align*}
respectively.
\end{definition}
Given a square matrix $C=(c_{ij})$
with entries in $\trop$, we denote by $\tper C$ the tropical
determinant of $C$ (i.e., the value of the optimal
assignment problem with weights $c_{ij}$).
We have the following characterization of the upper
dequantized tropical volume. 
\begin{theorem}\label{th-carac}
If $A\in \trop^{d\times m}$, then
\[
\qvol^+ A \ = \ \max_{I\subset [m], |I|=d} \tper A[I] \enspace ,
\]
where $A[I]$ denotes the maximal submatrix of $A$ obtained
by selecting the columns in $I$.
\end{theorem}
\begin{proof}
Consider an arbitrary lift $\puiseuxA$ of $A$. Let us choose a triangulation of the configuration of points determined by the columns of $\puiseuxA$. We identify this triangulation to a collection $\mathcal{J}$ of
subsets of $d+1$ elements of $[m]$, so that for every $J\in \mathcal{J}$,
the columns of $\puiseuxA$ of index in $J$ are the vertices
of precisely one simplex of this triangulation. Using{ the properties of the valuation,}
we observe that for all subsets $J\subset [m]$ with $d+1$ elements, we have $\val |\det \bar{\puiseuxA}[J]|\leq \tper \bar{A}[J]$.
Hence, we deduce from $\vol \puiseuxA
= \sum_{J\in \mathcal{J}} (d!)^{-1}|\det \bar\puiseuxA[J]|
\leq \sum_{J\subset[m],\;|J|=d+1} (d!)^{-1}|\det \bar\puiseuxA[J]|$
that $\val \vol \puiseuxA\leq \max_{J\subset[m],\;|J|=d+1} \tper \bar{A}[J]$.
Moreover, expanding the latter tropical determinant with respect to the first
row of $\bar{A}$, we get $\tper \bar{A}[J]= \max_{I\subset J, |I|=d} \tper A[I]$,
and so $\val \vol \puiseuxA\leq \max_{I\subset[m],\;|I|=d} \tper A[I]$.

To show the opposite inequality, we assume that the value
of the latter maximum differs from $-\infty$ (otherwise there is nothing to prove), and take $I$ such that $\tper A[I]$ reaches this
maximum. After reordering the columns of $A$, we may assume that $I=[d]$
and that $\tper\bar{A}[J]=\tper A[I]$ where $J:=[d+1]$.
We choose $\sigma$, a maximizing permutation
in the optimal assignment problem associated to $A[I]$. We now
choose a lift $\puiseuxA=(\a_{ij})$ of $A$, such that 
$\a_{ij}=t^{a_{ij}}$ if $j\neq\sigma(i)$, and $\a_{i\sigma(i)}=N t^{a_{i\sigma(i)}}$,
where $N>1$ is a fixed parameter. We have $|\det \bar{\puiseuxA}[J]|
\geq |N^{d}t^{\tper A[I]}|- |p|$, where $\p$ is the sum
of monomials of the form $\pm \prod_i \a_{i\pi(i)}$, over all permutations
$\pi$ of $[d+1]$ that differ from $\sigma$. Each of these monomials 
can be written as $\pm bt^v$
where $v\leq \tper A[I]$ and $0\leq b\leq N^{d-1}$. 
Choosing $N$ sufficiently large
($N\geq (d+1)!$ suffices),
we deduce that the leading exponent of $|\det \bar{\puiseuxA}[J]|$
is still $\tper A[I]$, i.e., 
$\val |\det \bar{\puiseuxA}[J]| = \tper {A}[I]$,
and so, $\val \vol \puiseuxA \geq \tper {A}[I]$.
\end{proof}
We shall say that a matrix $A$ is 
\emph{tropically sign-generic} if, in the optimal assignment problem associated to any maximal square submatrix of $A$, all the optimal permutations have the same
parity. 
\begin{theoremdefinition}[Dequantized tropical volume]
If $A\in \trop^{d\times m}$ is such that $\bar{A}$ is tropically sign generic,
then, $\qvol^+ A=\qvol^-A$, and we denote by $\qvol A$, 
the \emph{dequantized tropical volume} of $A$, this common value.
Moreover, for any lift $\puiseuxA$ of $A$, we have
\begin{equation}
\lim_{t\to\infty} \frac{\log \vol \conv \puiseuxA(t) }{\log t} \ = \ \qvol A
\enspace .
\label{e-limvol}
\end{equation}
\end{theoremdefinition}
\begin{proof}
It follows from the proof of Theorem~\ref{th-carac} that we can choose $J\subset [m]$, $|J|=d+1$ such that $\qvol^+A = \tper \bar{A}[J]$. Since $\bar{A}$ is tropically sign generic, for any lift $\puiseuxA$ of $A$, all the monomials of maximal valuation in the determinant expansion of $\det \bar{\puiseuxA}[J]$ have the same sign, and so $\val \det \bar{\puiseuxA}[J]=\tper \bar{A}[J]$. It follows
that $\qvol^+ A=\qvol^-A$. The identity involving the limit is an immediate
translation of this fact.
\end{proof}
The next proposition shows that $\qvol A$ depends only of the affine tropical convex hull of $A$.
\begin{proposition}\label{prop-intrinsic}
Let $A\in \trop^{d\times m}$. Suppose that $A$ or $\bar{A}$ is tropically sign generic, that $B\in \trop^{d\times p}$, and that $\atconv(A)\subset \atconv(B)$.
Then $\qvol^+(A)\leq \qvol^+(B)$.
In particular, if $\atconv(A)=\atconv(B)$ and if $\bar{A},\bar{B}$ are both tropically sign generic, then $\qvol(A)=\qvol(B)$.
\end{proposition}
This proposition will allow us to define, 
for those tropical polytopes that can be written as $P=\atconv(A)$ with $A$ tropically sign generic, the dequantized volume
$\qvol(P):=\qvol(A)$.
\begin{proof}
Since $\atconv(A)\subset\atconv(B)$, we can write $A=B \odot C$, for some matrix $C\in\trop^{p\times m}$ whose entries are non-positive.
Let $I\subset [m]$ of cardinality $d$ be such that in the optimal
assignment problem with weight matrix $A[I]$, all the optimal
permutations have the same parity. 
Then, the tropical analogue of the Binet--Cauchy formula~\cite[Ex.~3.7]{agg09} 
yields $\tper A[I] = \max_K (\tper B[K] + \tper C[K,I])$, 
the maximum being taken over the $d$-element subsets $K\subset [p]$,
where $C[K,I]$ denotes the $K\times I$ submatrix of $C$.
Then, we deduce from the characterization of $\qvol^+$ in Theorem~\ref{th-carac} that if $A$ is tropically sign generic, then $\qvol^+ A\leq \qvol^+ B$. 

Suppose now that $\bar{A}$ (rather than $A$) is tropically sign generic. 
Arguing as in the proof of Theorem~\ref{th-carac}, we can
assume that $\qvol^+A = \tper \bar{A}[J]= \tper A[I]$ where $I=[d]$
and $J=[d+1]$. In the optimal assignment problem 
with weight matrix $\bar A[J]$, the optimal permutations have
the same parity. By expanding $\tper \bar{A}[J]$ with respect to the first
row, we see that the same property must hold for the optimal
assignment problem with weight matrix $A[I]$. Then, we conclude as
in the first part of the proof.
\end{proof}
\begin{example}
The tropical genericity condition cannot be dispensed with. Consider
\[
A=\left(\begin{matrix}0& 0 & 0\\ 0 & 0 & 0\end{matrix}\right)
\quad \text{and} \quad
B = \left(\begin{matrix} 0& -1 & -2\\ 0 & -2 & -4\end{matrix}\right)
\enspace. 
\]
We have $\atconv A\subset \atconv B$.
However, $\qvol^+A = 0$, whereas $\qvol^+ B=-1$.
\end{example}
\begin{corollary}
Let $A\in \trop^{d\times m}$, $P:=\atconv A$,
$B\in \trop^{d\times p}$, $Q:=\atconv B$, $C:=(A,B)$ and suppose that $\bar{C}$ is tropically sign generic.
Then
\[ \qvol\bigl(\atconv(P \cup Q)\bigr) \ = \ \max\bigl(\qvol(P),\qvol(Q)\bigr)
\enspace .\]
\end{corollary}
In other words, the dequantized tropical volume is an \emph{idempotent measure}{~\cite{kolomaslov,akian}.}
Dyer and Frieze \cite{DyerFrieze:1988} showed that computing the volume of a classical polytope given by its vertices is $\sharp$P hard.
This is in contrast with the tropical situation.
\begin{corollary}
Let $A=(a_{ij})\in \trop^{d\times m}$.
The upper dequantized tropical volume $\qvol^+ A $
can be computed in strongly polynomial time.
\end{corollary}
\begin{proof}
Define the bipartite graph, in which one color class
is $[d]$, the other color class is $[m]$, and the
edge set is $E:=\{(i,j)\mid i\in[d], j\in[m], 
a_{ij}>-\infty\}$. Consider the transportation
polytope $X$, consisting of those non-negative vectors
$x=(x_{ij})_{(i,j)\in E}$ such that for all $i\in [d]$,
$\sum_{(i,j)\in E} x_{ij}=1$ and for all $j\in [d]$,
$\sum_{(i,j)\in E}x_{ij}\leq 1$. The extreme points of this polytope
have integer entries. Hence, by Theorem~\ref{th-carac},
$\qvol^+A$ coincides with the value of the linear
program $\max \sum_{(i,j)\in E} a_{ij}x_{ij}, \; x\in X$.
This is an optimal transport problem, which can be solved in strongly polynomial time; see \cite[\S21.6]{Schrijver03:CO_A}
\end{proof}

The dequantized tropical volume can be used to 
bound the volume of ordinary polytopes.
Instead of considering the non-archimedean valuation $\val$ over $\K$,
we shall consider the archimedean valuation $\log|\cdot|$ over $\R$.
Given a matrix $A=(a_{ij})\in \R_{\geq 0}^{d\times m}$, we
denote by $\Log A$ the matrix obtained by applying
the archimedean valuation entrywise.
\begin{theorem}\label{th-compar}
Let $A=(a_{ij})\in \R_{\geq 0}^{d\times m}$.
Then
\begin{equation}
  \vol \conv A \ \leq \ \alpha (d+1) \exp(\qvol^+(\Log A)) \enspace ,\label{e-compar}
\end{equation}
where $\alpha$ is the number of maximal cells of an arbitrary triangulation of the point configuration given by the columns of $A$.
\end{theorem}
\begin{proof}
We assume, without loss of generality, that every column of $A$ is an extreme point of $\conv A$, and that no two columns of $A$ are equal.
Choose a triangulation of the configuration of points represented by the columns of $A$, with $\alpha$ maximal cells.
As in the proof of Theorem~\ref{th-carac}, we represent this triangulation by a collection $\mathcal{J}$ of subsets of $d+1$ elements of $[m]$.
We have 
\begin{align*}
\vol \conv A \ &= \ \sum_{J\in \mathcal{J}} (d!)^{-1}|\det \bar A[J]|\\
&\leq \  \sum_{J\in \mathcal{J}} (d!)^{-1} (d+1)!\exp(\tdet \Log \bar A[J]) \\
&\leq \ \alpha(d+1)\exp(\qvol^+ \Log A)
\enspace .
\end{align*}
\end{proof}
{The size of any triangulation of $A$, and thus also $\alpha$, is bounded by $O(m^{\lceil(d+1)/2\rceil})$; see \cite[Cor.~2.6.2]{Triangulations}.}

It is instructive to compare the dequantized tropical volume $\qvol^\pm$ with the tropical volume $\tvol$. 
When $A\in \trop^{(n-1)\times n}$, the quantities $\tvol \bar A$ and $\qvol^\pm A$
provide different ``measures'' of the singularity
of the matrix $\bar{A}$. Indeed, 
one can check that $\qvol^- A=\qvol^+ A$
if and only if $\bar A$ is tropically sign generic,
meaning that all maximizing permutations in $\tper \bar{A}$
have the same sign, whereas $\tvol \bar A>0$ if and only 
there is only one maximizing permutation.
Therefore, $\tvol \bar A>0$ implies that $\qvol^+ A=\qvol^- A$,
but not vice versa. The dequantized tropical volume
has several properties to be expected from a measure
on tropical polytopes (like being defined for the convex hull
of any number of points in general position and being an idempotent measure).
However, the isoperimetric inequality for the dequantized volume
may read $\qvol^+A \leq (n-1)\times \max_{ij}a_{ij}$.
This leads to more degenerate isoperimetric results,
since the matrices achieving the equality do not have such
a rigid structure as the maximizing matrices in Theorem~\ref{thm:iso}.

\section{Concluding remarks}

Let us finally point out some open problems and further directions for future research.

\begin{itemize}
\item It would be interesting to study the extension of the present notion of tropical volume, $\tvol$, to rectangular matrices.
\item It seems to be an interesting combinatorial question whether or not the combinatorial type of generic isodiametric polytropes in $\R^d/\R\1$ is unique for each $d\geq 5$.  This asks to determine how exactly the ordinary convex polytope Iso$(d)$ intersects the secondary fan of $\Delta_{d-1}\times\Delta_{d-1}$.
\item So far, we only considered tropical versions of the distance and of the full dimensional volume, it would be interesting to tropicalize the lower dimensional volumes. 
\item Getting the optimal constant in~\eqref{e-compar} is probably a difficult issue. 
\end{itemize}

\bibliographystyle{alphaabbr}
\bibliography{biblio,complem}

\end{document}